\documentclass[12pt]{amsart}
\usepackage{amssymb}

\usepackage[pagebackref=false,colorlinks,linkcolor=blue,citecolor=blue]{hyperref}

\usepackage{mathrsfs,graphicx,amsthm, mathtools}

\usepackage[all]{xy}
\usepackage{tikz}
\usepackage{amsmath}
\input xy
\xyoption{all}

\newtheorem{lemma}{Lemma}[section]
\newtheorem{corollary}[lemma]{Corollary}
\newtheorem{theorem}[lemma]{Theorem}
\newtheorem{proposition}[lemma]{Proposition}

\theoremstyle{definition}
\newtheorem{remark}[lemma]{Remark}
\newtheorem{definition}[lemma]{Definition}

\DeclareMathOperator{\Mod}{Mod}
\DeclareMathOperator{\modd}{mod}
\DeclareMathOperator{\Endd}{End}

\DeclareMathOperator{\Hom}{Hom}

\DeclareMathOperator{\supp}{supp}

\DeclareMathOperator{\ind}{ind}
\DeclareMathOperator{\rad}{rad}

\newtheorem*{theorem a*}{Theorem A}
\newtheorem*{corollary a*}{Corollary A}

\begin{document}

\title{Relations for Grothendieck groups of $n$-cluster tilting subcategories}

\author{Raziyeh Diyanatnezhad}
\address{Department of Pure Mathematics\\
Faculty of Mathematics and Statistics\\
University of Isfahan\\
P.O. Box: 81746-73441, Isfahan, Iran}
\email{r.diyanat@sci.ui.ac.ir}

\author{Alireza Nasr-Isfahani}
\address{Department of Pure Mathematics\\
Faculty of Mathematics and Statistics\\
University of Isfahan\\
P.O. Box: 81746-73441, Isfahan, Iran\\ and School of Mathematics, Institute for Research in Fundamental Sciences (IPM), P.O. Box: 19395-5746, Tehran, Iran}
\email{nasr$_{-}$a@sci.ui.ac.ir / nasr@ipm.ir}

\subjclass[2010]{{16G10}, {16E20}, {18F30}, {18A25}}

\keywords{Grothendieck group, $n$-cluster tilting subcategory, $n$-homological pair, functor category}

\begin{abstract}
Let $\Lambda$ be an artin algebra and $\mathcal{M}$ be an n-cluster tilting subcategory of $\modd\Lambda$. We show that $\mathcal{M}$ has an additive generator if and only if the n-almost split sequences form a basis for the relations for the Grothendieck group of $\mathcal{M}$ if and only if every effaceable functor $\mathcal{M}\rightarrow Ab$ has finite length. As a consequence we show that if $\modd\Lambda$ has n-cluster tilting subcategory of finite type then the n-almost split sequences form a basis for the relations for the Grothendieck group of $\Lambda$.
\end{abstract}

\maketitle

%%%%%%%%%%%%%%%%%%%%%%%%%%%%%%%%%%%%%%
\section{Introduction}
Let $\Lambda$ be an artin algebra and $\modd\Lambda$ be the abelian category of all finitely generated left $\Lambda$-modules. The free abelian group $\mathrm{K}_0(\modd\Lambda,0)$ with the basis the set of isomorphism classes $[M]$ of $\Lambda$-modules $M$ modulo the subgroup generated by elements of the form $[M\oplus N]-[M]-[N]$ is called the split Grothendieck group of $\modd\Lambda$. $\mathrm{K}_0(\modd\Lambda,0)$ modulo the subgroup generated by $[L]-[M]+[N]$ for every exact sequence $0\to L\to M\to N\to0$ is called the Grothendieck group $\mathrm{K}_0(\modd\Lambda)$ of $\modd\Lambda$. The isomorphism classes of indecomposable modules form a basis for the generators of $\mathrm{K}_0(\modd\Lambda)$. It is natural to ask about the existence of a basis for the relations of $\mathrm{K}_0(\modd\Lambda)$. Butler in \cite{Bu} proved that if $\Lambda$ is of finite type (i.e. there are finitely many indecomposable finitely generated left $\Lambda$-modules up to isomorphism) then the almost split sequences generate the relations of $\mathrm{K}_0(\modd\Lambda)$. Auslander in \cite{Au3} proved that the generators given by Butler are linearly independent and also he proved the converse of the Butler's theorem. Several authors studied the same question for special subcategories of $\modd\Lambda$ and for a more general class of categories other than $\modd\Lambda$. Let $R$ be a commutative Cohen-Macaulay complete ring. Auslander and Reiten in \cite{AR} and Hiramatsu in \cite{H} studied this question for the subcategory of Cohen-Macaulay $R$-modules of $\modd R$. Enomoto studied this question in the context of Quillen's exact categories \cite{Eno}. Let $\mathcal{C}$ be an essentially small abelian Krull-Schmidt category. An additive contravariant functor $F$ from $\mathcal{C}$ to the abelian category Ab of all abelian groups is called effaceable if it has a presentation $\Hom_\mathcal{C}(-,X)\xrightarrow {(-,f)}\Hom_\mathcal{C}(-,Y)\rightarrow F\rightarrow 0$ such that $f$ is an epimorphism. Krause and Vossieck in \cite{Kra} proved that every effaceable functor $\mathcal{C}\rightarrow Ab$ has  finite length if and only if the almost split sequences in $\mathcal{C}$ generate the relations of the Grothendieck group of $\mathcal{C}$. In this paper, we study the relations of the Grothendieck group of an $n$-cluster tilting subcategory of $\modd\Lambda$.

Let $n$ be a fixed positive integer. Iyama in \cite{I1, I2} introduced $n$-cluster tilting subcategories of abelian categories in order to construct a higher Auslander correspondence, which is a higher dimensional version of the Auslander correspondence between algebras of finite representation type and Auslander algebras (for more details see \cite{I3}). A higher dimensional analogue of the Auslander-Reiten theory was developed in \cite{I1, I2, I4}. Higher Auslander-Reiten theory has several connections to other areas, for example non-commutative algebraic geometry, combinatorics, categorification of cluster algebras, higher category theory and symplectic geometry. Let $\Lambda$ be an artin algebra and $\mathcal{M}$ be an n-cluster tilting subcategory of $\modd\Lambda$. A pair $(\Lambda, \mathcal{M})$ is called n-homological pair \cite{HJV}. An $n$-homological pair $(\Lambda,\mathcal{M})$ is called of finite type if the number of isomorphism classes of indecomposable objects in $\mathcal{M}$ is finite \cite{EN}. The question of finiteness of $n$-homological pairs for $n\geq 2$, which is among the first that have been asked by Iyama \cite{I3}, is still open. Up to now, all known $n$-homological pairs with $n\geq 2$ are of finite type.

Let $(\Lambda, \mathcal{M})$ be an n-homological pair. All short exact sequences in $\mathcal{M}$ are split, but there are nice exact sequences with $n+2$ terms which is called n-exact sequence (see section \ref{nex}). An special class of n-exact sequences which are called n-almost split sequences (see Definition \ref{def2.3}) plays an important role in higher Auslander-Reiten theory. Reid in \cite[Definition 1.3]{Re} defined the Grothendieck group of $\mathcal{M}$ as follows. Let $\mathrm{K}_0(\mathcal{M},0)$ be a free abelian group with the basis the set of isomorphism classes $[X]$ of $\Lambda$-modules $X\in \mathcal{M}$ modulo the subgroup generated by elements of the form $[X\oplus Y]-[X]-[Y]$. For $X\in \mathcal{M}$ we denote by $[X]_0$ the element in $\mathrm{K}_0(\mathcal{M},0)$ corresponding to $X$. $\mathrm{K}_0(\mathcal{M},0)$ modulo the subgroup generated by $\sum^{n+1}_{i=0}(-1)^i[X_i]_0$ for every n-exact sequence $0\to X_{n+1}\to\cdots \to X_{0}\to 0$ in $\mathcal{M}$ is called the Grothendieck group $\mathrm{K}_0(\mathcal{M})$ of $\mathcal{M}$. In this paper, we show that the Iyama's question has a positive answer if and only if the n-almost split sequences in $\mathcal{M}$ form a basis for the relations of $\mathrm{K}_0(\mathcal{M})$. More precisely, we prove the following theorem.

\begin{theorem a*}$($Theorem \ref{main3}$)$\label{A}
Let $(\Lambda,\mathcal{M})$ be an $n$-homological pair. Then the following statements are equivalent:
\begin{itemize}
\item[(1)] The elements $\displaystyle\sum_{i=0}^{n+1}(-1)^i[A_i]_0$, for all $n$-almost split sequences $0\to A_{n+1}\to\cdots\to A_0\to0$ in $\mathcal{M}$ form a basis for the relations of $\mathrm{K}_0(\mathcal{M})$.
\item[(2)] Every effaceable functor $F:\mathcal{M}\rightarrow\mathrm{Ab}$ has finite length.
\item[(3)] $(\Lambda,\mathcal{M})$ is of finite type.
\end{itemize}
\end{theorem a*}

Note that in the classical case, Krause and Vossieck's proof of \cite[Theorem 4.10]{Kra} is close to the Butler's original proof of \cite[Theorem]{Bu} which is depended heavily on functorial arguments, but Auslander in the proof of Theorem 2.3 of \cite{Au3} used a bilinear form on $\mathrm{K}_0(\modd\Lambda,0)$ which has been used implicity by Benson and Parker in \cite{BP}. In higher dimensional case we need to use both techniques in the proof of Theorem A.

Almost all artin algebras $\Lambda$ are representation infinite (i.e. there are infinitely many indecomposable finitely generated left $\Lambda$-modules up to isomorphism). Let $\Lambda$ be a representation infinite artin algebra. It is natural to ask if there is a basis for the relations of $\mathrm{K}_0(\modd\Lambda)$. In the following corollary, we give a partial answer to this question.

\begin{corollary a*}$($Corollary \ref{cor}$)$\label{AA}
Let $\Lambda$ be an artin algebra. If $\modd\Lambda$ has an n-cluster tilting subcategory $\mathcal{M}$ of finite type then the elements $\displaystyle\sum_{i=0}^{n+1}(-1)^i[A_i]_0$, for all $n$-almost split sequences $0\to A_{n+1}\to\cdots\to A_0\to0$ in $\mathcal{M}$ form a basis for the relations of $\mathrm{K}_0(\modd\Lambda)$.
\end{corollary a*}

The paper is organized as follows. In section 2, we recall some definitions and known results that will
be needed in the rest of the paper. In section 3, first we show that for an n-homological pair $(\Lambda,\mathcal{M})$ if every effaceable functor $F:\mathcal{M}\rightarrow\mathrm{Ab}$ has finite length then $\mathrm{K}_0(\mathcal{M})\cong\mathrm{K}_0(\modd\Lambda)$ and finally we prove Theorem A and Corollary A.

\subsection{Notation}
Throughout this paper, we fix a positive integer $n$. We assume that $R$ is a commutative artinian ring and $\Lambda$ is an artin $R$-algebra. We denote by $\modd\Lambda$ the abelian category of all finitely generated left $\Lambda$-modules. For $M\in \modd\Lambda$, we denote by $\rad(M)$ the Jacobson radical of $M$. It is known that there exists only a finite number non-isomorphic simple $R$-modules $S_1,\ldots,S_l$. Let $J$ be the injective envelope of $\coprod_{i=1}^lS_i$. We denote by $D$ the duality $\Hom_R(-,J)$. Also we denote by $\mathcal{J}_\Lambda$ the Jacobson radical of $\modd\Lambda$, where for each $X,Y\in\modd\Lambda$
$$\mathcal{J}_\Lambda(X,Y)=\{h:X\rightarrow Y \mid 1_X-gh\;\text{is invertible for any}\;g:Y \rightarrow X \}.$$
We denote by $\mathrm{proj}(\mathcal{A})$  the full subcategory of an abelian category $\mathcal{A}$ consisting of finitely generated projective objects. Also we denote by $\ind(\mathcal{C})$ the full subcategory of an additive category $\mathcal{C}$ consisting of finitely generated indecomposable objects.

%%%%%%%%%%%%%%%%%%%%%%%%%%%%%%%%%%%%%%
\section{Preliminaries}
In this section, we collect some fundamental facts about $n$-abelian categories, $n$-cluster tilting subcategories and Grothendieck groups which will be used throughout the paper. For further details the readers are referred to \cite{I1, I3, J, Re}.
\subsection{$n$-Abelian categories}\label{nex}
Let $\mathcal{M}$ be an additive category and $d_{n+1}:X_{n+1} \rightarrow X_n$ be a morphism in $\mathcal{M}$. An $n$-cokernel of $d_{n+1}$ is a sequence
\begin{equation}
(d_n, \ldots, d_1): X_n \overset{d_n}{\longrightarrow} X_{n-1} \xrightarrow {d_{n-1}}\cdots \overset{d_{2}}{\longrightarrow} X_1 \overset{d_1}{\longrightarrow} X_{0} \notag
\end{equation}
in $\mathcal{M}$ such that for all $Y\in \mathcal{M}$
the induced sequence of abelian groups
\begin{align}
0 \rightarrow \Hom_{\mathcal{M}}(X_0,Y) \rightarrow \Hom_{\mathcal{M}}(X_1,Y) \rightarrow\cdots\rightarrow \Hom_{\mathcal{M}}(X_n,Y) \rightarrow \Hom_{\mathcal{M}}(X_{n+1},Y) \notag
\end{align}
is exact \cite[Definition 2.2]{J}. Dually, $n$-kernel of a morphism in $\mathcal{M}$ is defined. Also, a complex
\begin{equation}
X_{n+1} \xrightarrow {d_{n+1}} X_n \overset{d_n}{\longrightarrow} X_{n-1} \xrightarrow {d_{n-1}}\cdots \overset{d_{2}}{\longrightarrow} X_1 \overset{d_1}{\longrightarrow} X_{0} \notag
\end{equation}
in $\mathcal{M}$ is called an $n$-exact sequence if $(d_n, \ldots, d_1)$ is an $n$-cokernel of $d_{n+1}$ and $(d_{n+1}, \ldots, d_2)$ is an $n$-kernel of $d_1$ \cite[Definition 2.4]{J}.

\begin{definition}$($\cite[Definition 3.1]{J}$)$
An additive category $\mathcal{M}$ is called $n$-abelian if it satisfies the following axioms:
\begin{itemize}
\item[(A0)]
$\mathcal{M}$ is idempotent complete.
\item[(A1)]
$\mathcal{M}$ is closed under n-kernels and n-cokernels.
\item[(A2)]
For every monomorphism $d_{n+1}:X_{n+1} \rightarrow X_n$ in $\mathcal{M}$ and for every $n$-cokernel $(d_n, \ldots, d_1)$ of $d_{n+1}$, the  following sequence is $n$-exact:
\begin{equation}
X_{n+1} \xrightarrow {d_{n+1}} X_n \overset{d_n}{\longrightarrow} X_{n-1} \xrightarrow {d_{n-1}}\cdots \overset{d_{2}}{\longrightarrow} X_1 \overset{d_1}{\longrightarrow} X_{0}. \notag
\end{equation}
\item[(A3)]
For every epimorphism $d_{1}:X_{1} \rightarrow X_0$ in $\mathcal{M}$ and for every $n$-kernel $(d_{n+1}, \ldots, d_2)$ of $d_1$, the  following sequence is $n$-exact:
\begin{equation}
X_{n+1} \xrightarrow {d_{n+1}} X_n \overset{d_n}{\longrightarrow} X_{n-1} \xrightarrow {d_{n-1}}\cdots \overset{d_{2}}{\longrightarrow} X_1 \overset{d_1}{\longrightarrow} X_{0}. \notag
\end{equation}
\end{itemize}
\end{definition}

A subcategory $\mathcal{M}$ of $\modd\Lambda$ is called contravariantly finite if for every $A\in \modd\Lambda$ there exist an object $M\in \mathcal{M}$ and a morphism $f:M \rightarrow A$ such that for each $N\in \mathcal{M}$ the sequence of abelian groups
$$\Hom_{\Lambda}(N, M) \rightarrow \Hom_{\Lambda}(N, A)\rightarrow 0$$
is exact. Such a morphism $f$ is called a right $\mathcal{M}$-approximation of $A$. The notion of covariantly finite subcategory and left $\mathcal{M}$-approximation is defined dually. A functorially finite subcategory of $\modd\Lambda$ is a subcategory which is both covariantly and contravariantly finite in $\modd\Lambda$ \cite{AS}.

\begin{definition}$($\cite[Definition 2.2]{I1}$)$
A full subcategory $\mathcal{M}$ of $\modd\Lambda$ is called $n$-cluster tilting if it is functorially finite and
\begin{align*}
\mathcal{M}&=\{X\in\text{$\modd\Lambda$}\,|\,\mathrm{Ext}^i_\Lambda(X,\mathcal{M})=0 \,\,\, \text{for } 0<i<n\}\\
&=\{X\in\text{$\modd\Lambda$}\,|\,\mathrm{Ext}^i_\Lambda(\mathcal{M},X)=0 \,\,\, \text{for } 0<i<n\}.
\end{align*}
\end{definition}
It is clear that $\mathrm{proj}(\modd\Lambda)\subseteq \mathcal{M}$ and for all $X\in\modd\Lambda$, every right $\mathcal{M}$-approximation of $X$ is an epimorphism. Also, according to \cite[Theorem 3.16]{J}, any $n$-cluster tilting subcategory $\mathcal{M}$ of $\modd\Lambda$ is an $n$-abelian category.

 \begin{definition}$($\cite[Definition 2.5]{HJV} and \cite[Definition 2.13]{EN}$)$
Let $\Lambda$ be an artin algebra and $\mathcal{M}$ be an $n$-cluster tilting subcategory of $\modd\Lambda$. A pair $(\Lambda,\mathcal{M})$ is called an $n$-homological pair. An $n$-homological pair $(\Lambda,\mathcal{M})$ is called of finite type if $\mathcal{M}$ has an additive generator or equivalently the number of isomorphism classes of indecomposable objects in $\mathcal{M}$ is finite.
\end{definition}
We recall the following useful result.
 \begin{theorem}$($\cite[Theorem 2.2.3]{I1}$)$\label{reso}
Assume that $\mathcal{M}$ is an $n$-cluster tilting subcategory of $\modd\Lambda$. For any $X\in$ $\modd\Lambda$, there exists an exact sequence
\begin{equation}
0\rightarrow t_{n-1}\rightarrow\cdots\rightarrow t_0 \rightarrow X\rightarrow0, \notag
\end{equation}
with $t_i\in \mathcal{M}$ such that the following sequence is exact on $\mathcal{M}$
 \begin{equation}
0\to\Hom_\Lambda(-,t_{n-1})\to\cdots\to\Hom_\Lambda(-,t_0)\to \Hom_\Lambda(-,X)\to 0.\notag
\end{equation}
\end{theorem}
We recall that this sequence is called a right $\mathcal{M}$-resolution for $X$.

Motivated by the concept of almost split sequences in $\modd\Lambda$, Iyama defined the concept of $n$-almost split sequences in an $n$-cluster tilting subcategory $\mathcal{M}$ of $\modd\Lambda$ as follows.
\begin{definition}$($\cite[Definition 3.1]{I1}$)$\label{def2.3}
Let $\mathcal{M}$ be an $n$-cluster tilting subcategory of $\modd\Lambda$. An exact sequence
\begin{equation}
0\to A_{n+1}\xrightarrow {f_{n+1}}A_n\to\cdots\to A_1{\stackrel{f_1}{\longrightarrow}}A_0\to 0 \notag
\end{equation}
 with terms in $\mathcal{M}$ and $f_i\in \mathcal{J}_\Lambda$ for any $i$
 is called an $n$-almost split sequence if the sequence
 \begin{equation}
0\to\Hom_\Lambda(-,A_{n+1})\to\cdots\to\Hom_\Lambda(-,A_1)\to \mathcal{J}_\Lambda(-,A_0)\to 0\notag
\end{equation}
is exact on $\mathcal{M}$.
\end{definition}
It is obvious that every $n$-almost split sequence in $\mathcal{M}$ is an $n$-exact sequence.

The following theorem states the existence of $n$-almost split sequences in $n$-cluster tilting subcategories of $\modd\Lambda$.
\begin{theorem}$($\cite[Theorem 3.3.1]{I1}$)$\label{n-as}
Suppose that $(\Lambda,\mathcal{M})$ is an $n$-homological pair.
\begin{itemize}
\item[$(1)$]
 For any non-projective $X\in\ind(\mathcal{M})$, there exists an $n$-almost split sequence
\begin{equation}
0\to A_{n+1}\xrightarrow {f_{n+1}}A_n\to\cdots\to A_1{\stackrel{f_1}{\longrightarrow}}X\to 0 \notag
\end{equation}
in $\mathcal{M}$.
\item[$(2)$]
For any non-injective $Y\in\ind(\mathcal{M})$, there exists an $n$-almost split sequence
\begin{equation}
0\to Y\xrightarrow {f_{n+1}}A_n\to\cdots\to A_1{\stackrel{f_1}{\longrightarrow}}A_0\to 0 \notag
\end{equation}
in $\mathcal{M}$.
\end{itemize}
 \end{theorem}

\begin{definition}$($\cite[Definition 3.1]{JK}$)$
Let $(\Lambda,\mathcal{M})$ be an $n$-homological pair and
 $$\delta:\quad X_{n+1} \xrightarrow {d_{n+1}} X_n \overset{d_n}{\longrightarrow}\cdots \overset{d_{2}}{\longrightarrow} X_1 \overset{d_1}{\longrightarrow} X_{0}$$
 be an $n$-exact sequence in $\mathcal{M}$. The contravariant defect of $\delta$, denoted by $\delta^\ast$, is defined by the exact sequence of functors
 $$\Hom_\Lambda(-,X_1)\xrightarrow {\Hom_\Lambda(-, d_1)} \Hom_\Lambda(-,X_0) \rightarrow \delta^\ast\rightarrow0.$$
\end{definition}
Note that in case $n=1$, the contravariant defect of short exact sequences was introduced in \cite[Section $IV.4$]{ARS}.
\begin{remark}\label{def-sim}
By Definition \ref{def2.3}, for an $n$-almost split sequence
\begin{equation}
\delta:\quad 0\to A_{n+1}\xrightarrow {f_{n+1}}A_n\to\cdots\to A_1{\stackrel{f_1}{\longrightarrow}}A_0\to 0 \notag
\end{equation}
in $n$-cluster tilting subcategory $\mathcal{M}$ of $\modd\Lambda$, there exists the following exact sequence
\begin{align*}
0 \rightarrow \Hom_\Lambda(-,A_{n+1}) \rightarrow\cdots\Hom_\Lambda(-,A_1)\rightarrow \Hom_\Lambda(-,A_0) \rightarrow\Hom_\Lambda(-,A_0)/\mathcal{J}_\Lambda(-,A_0)\rightarrow0,
\end{align*}
on $\mathcal{M}$.
So $\delta^*=\Hom_\Lambda(-,A_0)/\mathcal{J}_\Lambda(-,A_0)$.
When $A_0$ is an indecomposable module, for any $X\in\ind(\mathcal{M})$,

$$\delta^*(X)=\left\{\begin{array}{lll}
0  \quad X\ncong A_0,\\
\\
\Endd_\Lambda(A_{0})/\mathcal{J}_\Lambda(A_0,A_{0})  \quad X\cong A_{0}.
\end{array}\right.$$
\end{remark}
\subsection{ Grothendieck groups }
Let $\mathcal{C}$ be an essentially small additive category and $G(\mathcal{C})$ be the free abelian group with basis the isomorphism classes of all objects $C$ in $\mathcal{C}$. Then the split Grothendieck group of $\mathcal{C}$ is defined as
\begin{equation}
\mathrm{K}_0(\mathcal{C},0)\coloneqq G(\mathcal{C})/\langle[A\oplus B]-[A]-[B]\ | \ A,B\in \mathcal{C}\rangle. \notag
\end{equation}
For any $X\in\mathcal{C}$, we denote by $[X]_0$ the element in $\mathrm{K}_0(\mathcal{C},0)$ corresponding to $X$.
Now, we assume that $\mathcal{C}$ is an abelian category. The Grothendieck group of $\mathcal{C}$ is a quotient group of $\mathrm{K}_0(\mathcal{C},0)$. Indeed, by considering the subgroup
$$\langle[X]_0-[Y]_0+[Z]_0\,|\,0\to X\to Y\to Z\to0\text{ is a short exact sequence in\,\,} \mathcal{C}\rangle,$$
the Grothendieck group of  $\mathcal{C}$ is defined as
\begin{equation}
\mathrm{K}_0(\mathcal{C})\coloneqq\mathrm{K}_0(\mathcal{C},0)/\langle[X]_0-[Y]_0+[Z]_0\,|\,0\to X\to Y\to Z\to0\text{ is a short exact sequence in\,\,} \mathcal{C}\rangle.\notag
\end{equation}
Recently, the Grothendieck group of $n$-abelian categories is defined as follows.
\begin{definition}$($\cite[Definition 1.3]{Re}$)$
Let $\mathcal{C}$ be an $n$-abelian category. Then the Grothendieck group of $\mathcal{C}$ is defined as
\begin{equation}
\mathrm{K}_0(\mathcal{C})\coloneqq\mathrm{K}_0(\mathcal{C},0)/\langle\sum^{n+1}_{i=0}(-1)^i[X_i]_0\,|\,0\to X_{n+1}\to\cdots \to X_{0}\to 0 \text{\ is an $n$-exact sequence in $\mathcal{C}$}\rangle. \notag
\end{equation}
\end{definition}
If $\mathcal{C}$ is an abelian or $n$-abelian category, then for any $X\in\mathcal{C}$ we denote by $[X]$ the element in $\mathrm{K}_0(\mathcal{C})$ corresponding to $X$.

Let $(\Lambda,\mathcal{M})$ be an $n$-homological pair. Recall that any $X\in\modd\Lambda$ has a right $\mathcal{M}$-resolution
\begin{equation}
0\rightarrow t_{n-1}^X\rightarrow\cdots\rightarrow t_0^X \rightarrow X\rightarrow0. \notag
\end{equation}
The index with respect to $\mathcal{M}$ is defined in \cite[Definition 1.4]{Re} by the following map
\begin{align*}
\mathrm{Ind}_{\mathcal{M}}:&\modd\Lambda\to\mathrm{K}_0(\mathcal{M},0).\\
& X\mapsto \sum^{n-1}_{i=0}(-1)^i[t_i^X]_0
\end{align*}
Note that $\mathrm{Ind}_{\mathcal{M}}$ is well defined (see \cite[ Remark 2.1]{Re}).

\section{Relations for Grothendieck groups}

In this section, we describe the relations of the Grothendieck group $\mathrm{K}_0(\mathcal{M})$ in case an $n$-homological pair $(\Lambda,\mathcal{M})$ is of finite type. By using $n$-almost split sequences we give a natural basis for $\mathrm{Ker}(\pi)$, where $\pi:\mathrm{K}_0(\mathcal{M},0)\rightarrow \mathrm{K}_0(\mathcal{M})$ is the natural projection.

Let $\mathcal{A}$ be an essentially small additive category. An additive contravariant functor $F$ from $\mathcal{A}$ to the abelian category Ab of all abelian groups is called $\mathcal{A}$-module. The category of all $\mathcal{A}$-modules is denoted by $\Mod(\mathcal{A})$. Morphisms in $\Mod(\mathcal{A})$ are natural transformations of contravariant functors. It is known that $\Hom_\mathcal{A}(-,X)$ is a projective object in $\Mod(\mathcal{A})$, for any $X\in\mathcal{A}$. An $\mathcal{A}$-module $F:\mathcal{A}\rightarrow \mathrm{Ab}$ is called finitely presented if it can be presented as
\begin{equation}
\Hom_\mathcal{A}(-,X)\xrightarrow {(-,f)}\Hom_\mathcal{A}(-,Y)\rightarrow F\rightarrow 0. \label{fp}
\end{equation}
The full subcategory of finitely presented $\mathcal{A}$-modules is denoted by $\mathrm{fp}(\mathcal{A})$.
For any $F\in\mathrm{fp}(\mathcal{A})$, if $f:X\rightarrow Y$ in presentation \eqref{fp} is an epimorphism, then $F$ is called effaceable. All effaceable functors form a full subcategory of $\mathrm{fp}(\mathcal{A})$ which is denoted by $\mathrm{eff}(\mathcal{A})$. Let $\mathcal{B}$ be a full subcategory of $\mathcal{A}$ and $F\in \Mod(\mathcal{A})$. We denote the restriction of $F$ to $\mathcal{B}$ by $F|_\mathcal{B}$. We refer the readers to \cite{Au, Au1} for more information about functor categories.

Let $(\Lambda,\mathcal{M})$ be an $n$-homological pair. Since $\mathcal{M}$ is closed under weak cokernels, $\mathrm{fp}(\mathcal{M})$ is an abelian category. By \cite[Lemma 8.5]{Bel}, there exists a fully faithful functor
\begin{align*}
\mathbb{H}:\modd\Lambda & \rightarrow \mathrm{fp}(\mathcal{M}),\\
 A&\mapsto \Hom_\Lambda(-,A)|_\mathcal{M}
\end{align*}
that induces an equivalence $\mathcal{M}\cong\mathrm{proj(fp(}\mathcal{M}))$.

For proving our main result we need the following proposition.

\begin{proposition}\label{f1}
Suppose that $(\Lambda,\mathcal{M})$ is an $n$-homological pair. Then the groups $\mathrm{K}_0(\mathcal{M},0)$ and $\mathrm{K}_0(\mathrm{fp}(\mathcal{M}))$ are isomorphic.
\end{proposition}
\begin{proof}
The equivalence $\mathcal{M}\cong\mathrm{proj(fp(}\mathcal{M}))$ induces a group isomorphism
\begin{align*}
\mathbb{H}:\mathrm{K}_0(\mathcal{M},0) & \rightarrow \mathrm{K}_0(\mathrm{proj(fp(}\mathcal{M})),0).\\
 [A]_0&\mapsto [\Hom_\Lambda(-,A)|_\mathcal{M}]_0
\end{align*}
By \cite[Theorem 8.23]{Bel}, $\mathrm{gl.dim(fp(}\mathcal{M}))\leq n+1$. Then by \cite[Theorem 4.6 of Chapter 8]{Bas}, there exists a group isomorphism
\begin{align*}
\mathbb{\eta}:\mathrm{K}_0(\mathrm{proj(fp(}\mathcal{M})),0) & \rightarrow \mathrm{K}_0(\mathrm{fp}(\mathcal{M})).\\
 [P]_0&\mapsto [P]
\end{align*}
Therefore $\sigma\coloneqq \eta \circ\mathbb{H}$ is a desired group isomorphism.
\end{proof}

The following easy lemma, which immediately follows from definitions, indicates the relationship between contravariant defect of $n$-exact sequences in $\mathcal{M}$ and effaceable functors in $\mathrm{fp}(\mathcal{M})$.
\begin{lemma}\label{lem1}
Let $(\Lambda,\mathcal{M})$ be an $n$-homological pair. Then
$$\mathrm{eff}(\mathcal{M})=\{\delta^\ast\,|\, \delta \text{\,\,is an n-exact sequence in\,\,} \mathcal{M}\}.$$
\end{lemma}

Note that a direct sum of $n$-almost split sequences and a direct summand of an $n$-almost split sequence are again $n$-almost split sequences. By
\cite[Proposition 3.1.1]{I1} the study of $n$-almost split sequences is reduced to that of $n$-almost split sequences with indecomposable end terms. From now on, by $n$-almost split sequences we means $n$-almost split sequences with indecomposable left and right terms.

\begin{lemma}\label{lem2}
Let $(\Lambda,\mathcal{M})$ be an $n$-homological pair and $\pi:\mathrm{K}_0(\mathcal{M},0)\rightarrow \mathrm{K}_0(\mathcal{M})$ be the natural projection. Then $\mathrm{Ker}(\pi)$ is generated by elements $\displaystyle\sum_{i=0}^{n+1}(-1)^i[A_i]_{0}$, for all $n$-almost split sequences $0\to A_{n+1}\to\cdots\to A_0\to0$ in $\mathcal{M}$ if and only if  for every $F\in\mathrm{eff}(\mathcal{M})$,
$[F]\in\langle[\delta^\ast]\,|\,\delta \text{\,\,is an n-almost split sequence in\,\,} \mathcal{M}\rangle$.
\end{lemma}
\begin{proof}
By Lemma \ref{lem1}, we can identify $\mathrm{eff}(\mathcal{M})$ with $\{\delta^\ast\,|\, \delta$ is an $n$-exact sequence in $\mathcal{M}\}$.
Let $\delta_{X_0}:0\to X_{n+1}\to\cdots\to X_0\to0$ be an $n$-exact sequence in $\mathcal{M}$. By Proposition \ref{f1}, we have
\begin{eqnarray}
\displaystyle\sum_{i=0}^{n+1}(-1)^i[X_i]_0\in\langle\displaystyle\sum_{i=0}^{n+1}(-1)^i[A_i]_0\,|\,0\to A_{n+1}\to\cdots\to A_0\to0\nonumber\\
\text{\,\,is an $n$-almost split sequence in\,\,} \mathcal{M}\rangle \nonumber
\end{eqnarray}
if and only if
\begin{eqnarray}
\sigma(\displaystyle\sum_{i=0}^{n+1}(-1)^i[X_i]_0)\in\langle\sigma(\displaystyle\sum_{i=0}^{n+1}(-1)^i[A_i]_0)\,|\,0\to A_{n+1}\to\cdots\to  A_0\to0\nonumber\\
\text{\,\,is an $n$-almost split sequence in\,\,} \mathcal{M}\rangle \nonumber
\end{eqnarray}
, where $\sigma$ is the group isomorphism in the proof of the Proposition \ref{f1}, if and only if
\begin{eqnarray}
\displaystyle\sum_{i=0}^{n+1}(-1)^i[\mathrm{Hom}_\Lambda(-,X_i)]\in\langle\displaystyle\sum_{i=0}^{n+1}(-1)^i[\mathrm{Hom}_\Lambda(-,A_i)]\,|\,0\to A_{n+1}\to\cdots\to A_0\to0\nonumber\\
\text{\,\,is an $n$-almost split sequence in\,\,} \mathcal{M}\rangle \nonumber
\end{eqnarray}
if and only if
$$[\delta^\ast_{X_0}]\in\langle[\delta^\ast]\,|\,\delta \text{\,\,is an n-almost split sequence in\,\,} \mathcal{M}\rangle.$$
\end{proof}

Let $\mathcal{C}$ be a Krull-Schmidt essentially small additive category and
$F:\mathcal{C}\rightarrow\mathrm{Ab}$ be an additive functor. $F$ is called of finite length if it has a finite composition series. The support of $F$ is a set of all indecomposable objects $X\in\mathcal{C}$ that $F(X)\neq0$ and is denoted by $\supp(F)$. There exists a relationship between $\supp(F)$ and length of $F$. We recall the following lemmas which are useful in the rest of the paper.

\begin{lemma}$($\cite[Lemma 3.4]{Kra}$)$\label{kra3.4}
Let $\mathcal{C}$ be Hom-finite and $F:\mathcal{C}\rightarrow\mathrm{Ab}$ a finitely generated functor. Then the
length of $F$ is finite if and only if $\supp(F)$ is finite.
\end{lemma}

\begin{lemma}$($\cite[Lemma 3.6]{Kra}$)$\label{kra3.6}
Let $\mathcal{C}$ be Hom-finite. Given functors $F$ and $(F_i)_{i\in I}$ in $\mathrm{fp}(\mathcal{C})$,
$$
[F]\in\langle[F_i] \,|\, i \in I\rangle \subseteq \mathrm{K}_0(\mathrm{fp}(\mathcal{C})) \text{\qquad implies\qquad} \supp(F)\subseteq \displaystyle\bigcup_{i\in I}\supp(F_i).$$
\end{lemma}

Let $\mathcal{A}$ be an abelian category. A subcategory $\mathcal{C}$ of $\mathcal{A}$ is called a Serre subcategory if for
any exact sequence
\begin{equation}
0\rightarrow A_1\rightarrow A_2\rightarrow A_3\rightarrow 0, \notag
\end{equation}
we have that $A_2\in \mathcal{C}$ if and only if $A_1\in \mathcal{C}$ and $A_3\in \mathcal{C}$.

\begin{proposition}\label{main1}
Let $(\Lambda,\mathcal{M})$ be an $n$-homological pair and $\pi:\mathrm{K}_0(\mathcal{M},0)\rightarrow \mathrm{K}_0(\mathcal{M})$ be the natural projection. Then the following statements are equivalent:
\begin{itemize}
\item[(1)]
Every effaceable functor $F:\mathcal{M}\rightarrow\mathrm{Ab}$ has finite length.
\item[(2)]
 $\mathrm{Ker}(\pi)$ is generated by elements $\displaystyle\sum_{i=0}^{n+1}(-1)^i[A_i]_{0}$, for all $n$-almost split sequences $0\to A_{n+1}\to\cdots\to A_0\to0$ in $\mathcal{M}$.
\end{itemize}
\end{proposition}
\begin{proof}
$(1)\Rightarrow(2):$ Let $F\in\mathrm{eff}(\mathcal{M})$. By assumption, $F$ has a finite composition series. Let $S$ be a simple composition factor of $F$. Since by \cite[Theorem 4.4]{EN1}, $\mathrm{eff}(\mathcal{M})$ is a Serre subcategory of $\mathrm{fp}(\mathcal{M})$, $S\in \mathrm{eff}(\mathcal{M})$. Then there exists an $n$-almost split sequence $\delta:0\to A_{n+1}\to\cdots\to A_0\to0$ that $\delta^\ast=S$. So
$$[F]\in\langle[\delta^\ast]\,|\,\delta \text{\,\,is an $n$-almost split sequence in\,\,} \mathcal{M} \rangle.$$ Therefore the result follows from Lemma \ref{lem2}.

$(2)\Rightarrow(1):$ Let $F\in\mathrm{eff}(\mathcal{M})$. Then by Lemma \ref{lem2}, $F=\displaystyle\sum_{i=1}^{t}\lambda_i[\delta^\ast_i]$ where $\delta^\ast_i$ is a contravariant defect of an $n$-almost split
sequence $\delta_i:0\to A_{n+1}^i\to\cdots\to A_0^i\to0$ and $\lambda_i\in\mathbb{Z}$, for each $i=1,\ldots,t$. So $F\in\langle[\delta^\ast_i]\,|\,i=1,\ldots,t \rangle$. By Lemma \ref{kra3.6}, we have $\supp(F)\subseteq \displaystyle\bigcup_{i=1}^t\supp(\delta^\ast_i)$. On the other hand, we know that $\supp(\delta^\ast_i)=\{A_0^i\}$ for each $i=1,\ldots,t$. Thus $\supp(F)$ is finite and by Lemma \ref{kra3.4}, $F$ has finite length.
\end{proof}

For two modules $A$ and $B$ in $\modd\Lambda$, we denote by $\langle A,B\rangle$ the length of $\mathrm{Hom}_\Lambda(A,B)$ as $R$-module and we consider bilinear form $\langle\, ,\,\rangle:\mathrm{K}_0(\modd\Lambda,0)\times\mathrm{K}_0(\modd\Lambda,0)\rightarrow \mathbb{Z}$. Let $(\Lambda,\mathcal{M})$ be an $n$-homological pair.
We associate with every $A\in\ind(\mathcal{M})$ an element $\beta_A$ in $\mathrm{K}_0(\mathcal{M},0)$, in the following way:

If $A$ is indecomposable non-projective, then by Theorem \ref{n-as}, there exists an $n$-almost split sequence
$$0\to A_{n+1}\to\cdots\to A_1\to A\to0,$$
in $\mathcal{M}$. We set $\beta_A\coloneqq\displaystyle\sum_{i=0}^{n+1}(-1)^i[A_i]_0$, where $A_0=A$.

If $A$ is indecomposable projective, then we set $\beta_A\coloneqq[A]_0-\mathrm{Ind}_\mathcal{M}(\rad(A))$.
\begin{lemma}\label{au1}
Suppose that $(\Lambda,\mathcal{M})$ is an $n$-homological pair and $A\in\ind(\mathcal{M})$. Then the following statements hold:
\begin{itemize}
\item[(1)]
Let $l_A$ be the length of $\Endd_\Lambda(A)/\mathcal{J}_\Lambda(A,A)$ as $R$-module. For every $X\in\ind(\mathcal{M})$, we have
$$
\langle[X]_0,\beta_A\rangle\coloneqq\left\{
\begin{array}{lll}
0 \quad X\ncong A,\\
\\
l_A  \quad X\cong A.
\end{array}\right. $$
\item[(2)]
$\{\beta_A\,|\,A\in\ind(\mathcal{M})\}$ is linearly independent in $\mathrm{K}_0(\mathcal{M},0)$.
\end{itemize}
\end{lemma}
\begin{proof}
(1) We have two cases. Case 1: Let $A$ be non-projective. Then there exists an $n$-almost split sequence
$$\delta_A:0\to A_{n+1}\to\cdots\to A_1\to A\to0,$$
in $\mathcal{M}$
 that for every $X\in\ind(\mathcal{M})$, we have $\langle[X]_0,\beta_A\rangle=\langle\delta^\ast_A(X)\rangle$, where $\langle\delta^\ast_A(X)\rangle$ is the length of $\delta^\ast_A(X)$ as $R$-module. Thus the result follows by Remark \ref{def-sim}. Case 2: Let $A$ be projective. We consider the right $\mathcal{M}$-resolution
$$0\to t_{n-1}\to\cdots\to t_1\to t_0\to\rad(A)\to0,$$
 for $\rad(A)$. By definition, we know that the following sequence is exact for any $X\in\ind(\mathcal{M})$,
 $$0\to\Hom_\Lambda(X,t_{n-1})\to\cdots\to\Hom_\Lambda(X,t_0)\to \Hom_\Lambda(X,\rad(A))\to 0.$$
Therefore $\langle[X]_0,\mathrm{Ind}_\mathcal{M}(\rad(A))\rangle=\langle[X]_0,[\rad(A)]_0\rangle$
and so for any $X\in\ind(\mathcal{M})$, we obtain
\begin{align*}
\langle[X]_0,\beta_A\rangle &=\langle[X]_0,[A]_0\rangle-\langle[X]_0,\mathrm{Ind}_\mathcal{M}(\rad(A))\rangle\\
&=\langle[X]_0,[A]_0\rangle-\langle[X]_0,[\rad(A)]_0\rangle.
\end{align*}
Assume that $X$ is not isomorphic to $A$. Since $\rad(A)$ is the unique maximal submodule of $A$, $\Hom_\Lambda(X,A)=\Hom_\Lambda(X,\rad(A))$ and therefore $\langle[X]_0,\beta_A\rangle=0$. Moreover, $\Hom_\Lambda(A,\rad(A))=\mathcal{J}_\Lambda(A,A)$  and hence $\langle[A]_0,\beta_A\rangle=l_A$.

(2)
Assume that $y=\displaystyle\sum_{i=1}^t\lambda_i\beta_{A_i}=0$ with $\lambda_i\in\mathbb{Z}$ and $A_i\in\ind(\mathcal{M})$, for each $1\leq i\leq t$. For each $j=1,\ldots,t$, we have
 $$\langle [A_j]_0,y\rangle= \displaystyle\sum_{i=1}^t\lambda_i\langle[A_j]_0,\beta_{A_i}\rangle=0.$$
Part (1) implies that
 $$\displaystyle\sum_{i=1}^t\lambda_i\langle[A_j]_0,\beta_{A_i}\rangle=\lambda_jl_{A_j}.$$
Then for each $j=1,\ldots,t$, $\lambda_jl_{A_j}=0$ and hence $\lambda_j=0$.
\end{proof}

\begin{lemma}\label{au2}
Let $(\Lambda,\mathcal{M})$ be an $n$-homological pair. If $\{\beta_A\,|\,A\in\ind(\mathcal{M})\}$ generates $\mathrm{K}_0(\mathcal{M},0)$, then $(\Lambda,\mathcal{M})$ is of finite type.
\end{lemma}
\begin{proof}
It is clear that $D\Lambda\in\mathcal{M}$. By assumption, $[D\Lambda]_0=\displaystyle\sum_{A\in\ind(\mathcal{M})}\lambda_A\beta_A$ that $\lambda_A\neq0$ for only finitely many $A\in\ind(\mathcal{M})$. Since every $X\in\ind(\mathcal{M})$ has injective envelope in $\mathcal{M}$, we have
$$\langle[X]_0,[D\Lambda]_0\rangle=\displaystyle\sum_{A\in\ind(\mathcal{M})}\lambda_A\langle[X]_0,\beta_A\rangle\neq0.$$
By part (1) of Lemma \ref{au1}, $\displaystyle\sum_{A\in\ind(\mathcal{M})}\lambda_A\langle[X]_0,\beta_A\rangle=\lambda_Xl_X$, for every $X\in\ind(\mathcal{M})$. Therefore for every $X\in\ind(\mathcal{M})$,  $\lambda_X\neq0$ and so the number of indecomposable objects in $\mathcal{M}$ is finite.
\end{proof}

Now we are ready to state our main result in this paper.

\begin{theorem}\label{main3}
Let $(\Lambda,\mathcal{M})$ be an $n$-homological pair. Then the following statements are equivalent:
\begin{itemize}
\item[(1)]
The elements $\displaystyle\sum_{i=0}^{n+1}(-1)^i[A_i]_0$, for all $n$-almost split sequences $0\to A_{n+1}\to\cdots\to A_0\to0$ in $\mathcal{M}$ form a basis for the kernel of $\pi:\mathrm{K}_0(\mathcal{M},0)\rightarrow \mathrm{K}_0(\mathcal{M})$.
\item[(2)]
Every effaceable functor $F:\mathcal{M}\rightarrow\mathrm{Ab}$ has finite length.
\item[(3)]
$(\Lambda,\mathcal{M})$ is of finite type.
\end{itemize}
\end{theorem}

For the proof of the implication $(2)\Rightarrow(3)$ in the above theorem we need to show that if every effaceable functor $F:\mathcal{M}\rightarrow\mathrm{Ab}$ has finite length then $\mathrm{K}_0(\mathcal{M})\cong\mathrm{K}_0(\modd\Lambda).$ Note that Reid in \cite[Theorem 3.7]{Re} proved that if an $n$-homological pair $(\Lambda,\mathcal{M})$ is of finite type, then $\mathrm{K}_0(\mathcal{M})\cong\mathrm{K}_0(\modd\Lambda)$. But in our case we do not have finiteness assumption for $n$-homological pair $(\Lambda,\mathcal{M})$. For proving this result we need the following proposition.

\begin{proposition}\label{f2}
Let $(\Lambda,\mathcal{M})$ be an $n$-homological pair. Then the following statements hold:
\begin{itemize}
\item[(1)]
If $F\in\mathrm{fp}(\modd\Lambda)$, then $F|_{\mathcal{M}}\in\mathrm{fp}(\mathcal{M})$.
\item[(2)]
If $F\in\mathrm{eff}(\modd\Lambda)$, then $F|_{\mathcal{M}}\in\mathrm{eff}(\mathcal{M})$.
\end{itemize}
\end{proposition}
\begin{proof}
$(1)$ Let $F\in\mathrm{fp}(\modd\Lambda)$. There exists a presentation
$$
\Hom_\Lambda(-,X)\xrightarrow {(-,f)}\Hom_\Lambda(-,Y)\rightarrow F\rightarrow 0.
$$
Specially, the sequence of functors
$$
\Hom_\Lambda(-,X)|_{\mathcal{M}}\xrightarrow {(-,f)}\Hom_\Lambda(-,Y)|_{\mathcal{M}}\rightarrow F|_{\mathcal{M}}\rightarrow0
$$
 is exact in $\Mod(\mathcal{M})$.
By Theorem \ref{reso}, for $X, Y\in\modd\Lambda$  there exist right $\mathcal{M}$-resolutions:
\begin{align*}
0\rightarrow t_{n-1}^X\rightarrow\cdots\rightarrow t_0^X \rightarrow X\rightarrow0,\\
0\rightarrow t_{n-1}^Y\rightarrow\cdots\rightarrow t_0^Y \rightarrow Y\rightarrow0,
\end{align*}
such that $t_i^X$ and $t_i^Y$ belong to $\mathcal{M}$, for $i=0,\ldots,n-1$ and the following sequences are exact:
\begin{align*}
0\to\Hom_\Lambda(-,t_{n-1}^X)|_{\mathcal{M}}\to\cdots\to\Hom_\Lambda(-,t_0^X)|_{\mathcal{M}}\to \Hom_\Lambda(-,X)|_{\mathcal{M}}\to 0,\\
0\to\Hom_\Lambda(-,t_{n-1}^Y)|_{\mathcal{M}}\to\cdots\to\Hom_\Lambda(-,t_0^Y)|_{\mathcal{M}}\to \Hom_\Lambda(-,Y)|_{\mathcal{M}}\to 0.
\end{align*}
So $\Hom_\Lambda(-,X)|_{\mathcal{M}}$  and $\Hom_\Lambda(-,Y)|_{\mathcal{M}}$ belong to $\mathrm{fp}(\mathcal{M})$ and by \cite[Proposition 4.2(b)]{Au1}, $F|_\mathcal{M}\in\mathrm{fp}(\mathcal{M})$.

$(2)$ Let $F\in\mathrm{eff}(\modd\Lambda)$. There exists a presentation
$$
\Hom_\Lambda(-,X)\xrightarrow {(-,f)}\Hom_\Lambda(-,Y)\rightarrow F\rightarrow 0,
$$
 with epimorphism $f$. By part (1), $F|_\mathcal{M}\in\mathrm{fp}(\mathcal{M})$. Then we have an exact sequence
$$\Hom_{\mathcal{M}}(-,t_2)\xrightarrow {(-,g)}\Hom_{\mathcal{M}}(-,t_1)\xrightarrow {\theta} F|_{\mathcal{M}}\rightarrow0,$$
 with $t_1,t_2\in\mathcal{M}$.
In order to complete the proof, it is enough to show that $g$ is an epimorphism.
By the Yoneda's lemma, there exists $x\in F|_{\mathcal{M}}(t_1)$ corresponding to the functorial morphism $\theta$. By assumption, $F\in\mathrm{eff}(\modd\Lambda)$. So for $t_1\in\mathcal{M}$ and $x\in F|_{\mathcal{M}}(t_1)=F(t_1)$, there exists an epimorphism $\psi:Z\rightarrow t_1$ in $\modd\Lambda$ which $F(\psi)(x)=0$.
Consider right $\mathcal{M}$-approximation $\varphi:t\rightarrow Z$ for $Z$. Then the composition $\psi\circ\varphi$ is an epimorphism in $\mathcal{M}$ and we have $F|_{\mathcal{M}}(\psi\circ\varphi)(x)=F(\psi\circ\varphi)(x)=F(\varphi)\circ F(\psi)(x)=0$. Therefore by the Yoneda's lemma, $\theta\circ(-,\psi\circ\varphi)=0$. Then $\mathrm{Im}(-,\psi\circ\varphi)\subseteq\mathrm{Ker}(\theta)$ and hence $\mathrm{Im}(-,\psi\circ\varphi)\subseteq\mathrm{Im}(-,g)$.
Now, consider the following diagram
\begin{center}
\begin{tikzpicture}
\node (X1) at (4,3) {$\Hom_\Lambda(-,t)|_{\mathcal{M}}$};
\node (X2) at (0,1) {$\Hom_\Lambda(-,t_2)|_{\mathcal{M}}$};
\node (X3) at (4,1) {$\mathrm{Im}(-,g)$};
\node (X4) at (7,1) {$0$.};
\draw [->,thick] (X1) -- (X3) node [midway,right] {$(-,\psi\circ\varphi)$};
\draw [->,thick] (X2) -- (X3) node [midway,above] {$(-,g)$};
\draw [->,thick] (X3) -- (X4) node [midway,above] {};

\end{tikzpicture}
\end{center}
$\Hom_\Lambda(-,t)|_{\mathcal{M}}\in \mathrm{proj(fp(}\mathcal{M}))$, so there exists $h:t\rightarrow t_2$ such that $g\circ h= \psi\circ\varphi$. Since $\psi\circ\varphi$ is an epimorphism in $\mathcal{M}$, $g$ is an epimorphism and the result follows.
\end{proof}

\begin{theorem}\label{eqofgro}
 Let $(\Lambda,\mathcal{M})$ be an $n$-homological pair. If every effaceable functor $F:\mathcal{M}\rightarrow \mathrm{Ab}$ has finite length, then $\mathrm{K}_0(\mathcal{M})\cong\mathrm{K}_0(\modd\Lambda).$
\end{theorem}
\begin{proof} Naturally, there exists a morphism
\begin{align*}
i:\mathrm{K}_0(\mathcal{M},0) & \rightarrow \mathrm{K}_0(\modd\Lambda,0).\\
 [X]_0&\mapsto [X]_0
\end{align*}
By \cite[Lemma 3.3]{Re}, $i$ induces a well-defined group homomorphism
 \begin{align*}
g:\mathrm{K}_0(\mathcal{M}) \rightarrow \mathrm{K}_0(\modd\Lambda).
\end{align*}
Let $\pi:\mathrm{K}_0(\mathcal{M},0)\rightarrow \mathrm{K}_0(\mathcal{M})$ be the natural projection. Consider the map
\begin{align*}
g^\prime:\mathrm{K}_0(\modd\Lambda)  \longrightarrow \mathrm{K}_0(\mathcal{M}).\\
 [X]\mapsto \mathrm{Ind}_\mathcal{M}(X)+\mathrm{Ker}(\pi)
\end{align*}
We show that $g^\prime$ is a well-defined group homomorphism. For this purpose, it is enough to show that for every short exact sequence
 $0\rightarrow X\rightarrow Y\rightarrow Z\rightarrow 0$
 in $\modd\Lambda$,
 $$\pi(\mathrm{Ind}_\mathcal{M}(X)-\mathrm{Ind}_\mathcal{M}(Y)+\mathrm{Ind}_\mathcal{M}(Z))=0.$$
Let $\delta:0\rightarrow X\rightarrow Y\rightarrow Z\rightarrow 0$ be a short exact sequence in $\modd\Lambda$. There exists an exact sequence
 \begin{align*}
0\to\Hom_\Lambda(-,X)\to\Hom_\Lambda(-,Y)\to \Hom_\Lambda(-,Z)\to\delta^\ast\to 0,
\end{align*}
where $\delta^\ast$ is the contravariant defect of $\delta$.
Specially, the sequence
 \begin{align*}
0\to\Hom_\Lambda(-,X)|_\mathcal{M}\to\Hom_\Lambda(-,Y)|_\mathcal{M}\to \Hom_\Lambda(-,Z)|_\mathcal{M}\to\delta^\ast|_\mathcal{M}\to 0,
\end{align*}
is exact. By Proposition \ref{f2}, this sequence lies in $\mathrm{fp}(\mathcal{M})$ and therefore
\begin{equation}\label{eq1}
[\delta^\ast|_\mathcal{M}]=[\Hom_\Lambda(-,X)|_\mathcal{M}]-[\Hom_\Lambda(-,Y)|_\mathcal{M}]+[\Hom_\Lambda(-,Z)|_\mathcal{M}],
\end{equation}
 in $\mathrm{K}_0(\mathrm{fp}(\mathcal{M}))$.
On the other hand, according to Theorem \ref{reso}, for any $A\in\modd\Lambda$ there exists the right $\mathcal{M}$-resolution
\begin{align*}
0\rightarrow t_{n-1}^A\rightarrow\cdots\rightarrow t_0^A \rightarrow A\rightarrow0,
\end{align*}
where $t_i^A\in \mathcal{M}$, for $i=0,\ldots,n-1$ and the following sequence is exact in $\mathrm{fp}(\mathcal{M})$
\begin{align*}
0\to\Hom_\Lambda(-,t_{n-1}^A)|_{\mathcal{M}}\to\cdots\to\Hom_\Lambda(-,t_0^A)|_{\mathcal{M}}\to \Hom_\Lambda(-,A)|_{\mathcal{M}}\to 0.
\end{align*}
Therefore we have
$$[\Hom_\Lambda(-,A)|_\mathcal{M}]=\displaystyle \sum_{i=0}^{n-1}(-1)^i[\Hom_\Lambda(-,t_i^A)|_\mathcal{M}]$$
in $\mathrm{K}_0(\mathrm{fp}(\mathcal{M}))$.
By \eqref{eq1} and Proposition \ref{f1}, we obtain
\begin{align*}
[\delta^\ast|_\mathcal{M}]&=\displaystyle \sum_{i=0}^{n-1}(-1)^i[\Hom_\Lambda(-,t_i^X)|_\mathcal{M}]-\displaystyle \sum_{i=0}^{n-1}(-1)^i[\Hom_\Lambda(-,t_i^Y)|_\mathcal{M}]+\displaystyle \sum_{i=0}^{n-1}(-1)^i[\Hom_\Lambda(-,t_i^Z)|_\mathcal{M}]\\
&=\sigma(\displaystyle \sum_{i=0}^{n-1}(-1)^i[t_i^X]_0)-\sigma(\displaystyle \sum_{i=0}^{n-1}(-1)^i[t_i^Y]_0)+\sigma(\displaystyle \sum_{i=0}^{n-1}(-1)^i[t_i^Z]_0)
\end{align*}
in $\mathrm{K}_0(\mathrm{fp}(\mathcal{M}))$. Consequently, we have the following identity in $\mathrm{K}_0(\mathrm{fp}(\mathcal{M}))$,
\begin{equation}\label{eq3}
[\delta^\ast|_\mathcal{M}]=\sigma(\mathrm{Ind}_\mathcal{M}(X)-\mathrm{Ind}_\mathcal{M}(Y)+\mathrm{Ind}_\mathcal{M}(Z)).
\end{equation}
Proposition \ref{f2} implies that $\delta^\ast|_{\mathcal{M}}\in\mathrm{eff}(\mathcal{M})$ and by assumption $\delta^\ast|_\mathcal{M}$ has finite length. Then we have
\begin{equation}\label{eq4}
[\delta^\ast|_\mathcal{M}]=[S_1]+\cdots+[S_t],
\end{equation}
 in $\mathrm{K}_0(\mathrm{fp}(\mathcal{M}))$ where $S_1,\ldots, S_t$ are simple composition factors of $\delta^\ast|_\mathcal{M}$. For each $i$, $S_i$ can be written as $\dfrac{\Hom_\Lambda(-,M_i)|_\mathcal{M}}{\mathcal{J}_\Lambda(-,M_i)|_\mathcal{M}}$, where $M_i\in\mathcal{M}$ is an indecomposable module and for any indecomposable module $X\in\mathcal{M}$ which is not isomorphic to $M_i$, $S_i(X)=0$.
Then for each $1\leq i\leq t$, $\delta^\ast|_\mathcal{M}(M_i)\neq0$. If $M_i$ is an indecomposable projective module for some $1\leq i\leq t$, then, by the Auslander's defect formula \cite [$\mathrm{IV}$, Theorem 4.1]{ARS}, we have $\delta^\ast|_\mathcal{M}(M_i)=\delta^\ast(M_i)=0$ which is a contradiction. Then for each $1\leq i\leq t$, $M_i$ is not projective and by Theorem \ref{n-as}, there exists an $n$-almost split sequence
\begin{align*}
0\rightarrow A_{n+1}^i\rightarrow\cdots\rightarrow A_1^i \rightarrow A_0^i\coloneqq M_i\rightarrow0,
\end{align*}
such that the sequence
\begin{align*}
0\to\Hom_\Lambda(-, A_{n+1}^i)|_{\mathcal{M}}\to\cdots\to\Hom_\Lambda(-, A_1^i)|_{\mathcal{M}}\to \Hom_\Lambda(-, M_i)|_{\mathcal{M}}\to S_i\to 0
\end{align*}
is exact in $\mathrm{fp}(\mathcal{M})$. Thus for each $1\leq i\leq t$, we have identities
\begin{equation}\label{eq5}
[S_i]=\displaystyle \sum_{j=0}^{n+1}(-1)^j[\Hom_\Lambda(-,A_j^i)|_{\mathcal{M}}]=\sigma(\displaystyle \sum_{j=0}^{n+1}(-1)^j[A_j^i]_0)
\end{equation}
in $\mathrm{K}_0(\mathrm{fp}(\mathcal{M}))$.
From \eqref{eq3}, \eqref{eq4} and \eqref{eq5} we obtain
\begin{align*}
\sigma(\mathrm{Ind}_\mathcal{M}(X)-\mathrm{Ind}_\mathcal{M}(Y)+\mathrm{Ind}_\mathcal{M}(Z))=\displaystyle \sum_{i=1}^{t}\sigma(\displaystyle \sum_{j=0}^{n+1}(-1)^j[A_j^i]_0).
\end{align*}
Since $\sigma$ is a group isomorphism,
\begin{align*}
\mathrm{Ind}_\mathcal{M}(X)-\mathrm{Ind}_\mathcal{M}(Y)+\mathrm{Ind}_\mathcal{M}(Z)=\displaystyle \sum_{i=1}^{t}(\displaystyle \sum_{j=0}^{n+1}(-1)^j[A_j^i]_0).
\end{align*}
 By applying $\pi$, we obtain
 $\pi(\mathrm{Ind}_\mathcal{M}(X)-\mathrm{Ind}_\mathcal{M}(Y)+\mathrm{Ind}_\mathcal{M}(Z))=0$ as required.
It is easy to see that $g^\prime$ and $g$ are mutually inverse and the assertion follows.
\end{proof}

We recall that $\mathrm{proj}(\modd\Lambda)\subseteq\mathcal{M}$ and we denote by $\mathscr{P}(\Lambda)$ the full subcategory of $\ind(\mathcal{M})$ consisting of all indecomposable projective $\Lambda$-modules.

Now, we are ready to prove Theorem \ref{main3}.
\begin{proof}[Proof of Theorem \ref{main3}]
$(1)\Leftrightarrow(2):$ Follows from Theorem \ref{main1} and part (2) of Lemma \ref{au1}.\\
$(3)\Rightarrow(2):$ Follows from Lemma \ref{kra3.4}.\\
$(2)\Rightarrow(3):$ By \cite[$\mathrm{I}$, Theorem 1.7]{ARS}, $\{[A]-[\rad(A)] \,|\, A\in\mathscr{P}(\Lambda)\}$ is a free basis for $\mathrm{K}_0(\modd\Lambda)$. By using the group isomorphism $g^\prime$ in the proof of Theorem \ref{eqofgro} we obtain a free basis $\{[A]_0-\mathrm{Ind}_\mathcal{M}(\rad(A))+\mathrm{Ker}(\pi)\,|\,A\in\mathscr{P}(\Lambda)\}$ for $\mathrm{K}_0(\mathcal{M})$. Then $\{\pi(\beta_A)\,|\,A\in\mathscr{P}(\Lambda)\}$ is a free basis for $\mathrm{K}_0(\mathcal{M})$. By Theorem \ref{main1}, $\mathrm{Ker}(\pi)$ is generated by $\{\beta_A\,|\,A\in\ind(\mathcal{M})\backslash\mathscr{P}(\Lambda)\}$. Therefore $\{\beta_A\,|\,A\in\ind(\mathcal{M})\}$ generate $\mathrm{K}_0(\mathcal{M},0)$ and the result follows from Lemma \ref{au2}.
\end{proof}

As a consequence of our main result, when $(\Lambda,\mathcal{M})$ is an $n$-homological pair of finite type, we give the following description of the relations of $\mathrm{K}_0(\modd\Lambda)$.

\begin{corollary}\label{cor}
Let $\Lambda$ be an artin algebra. If $\modd\Lambda$ has an $n$-cluster tilting subcategory $\mathcal{M}$ of finite type then the elements $\displaystyle\sum_{i=0}^{n+1}(-1)^i[A_i]_0$, for all $n$-almost split sequences $0\to A_{n+1}\to\cdots\to A_0\to0$ in $\mathcal{M}$ form a basis for the kernel of $\pi_\Lambda:\mathrm{K}_0(\modd\Lambda,0)\rightarrow \mathrm{K}_0(\modd\Lambda)$.
\end{corollary}
\begin{proof}
By using the proof of Theorem \ref{eqofgro}, we can see that $\mathrm{Ker}(\pi_\Lambda)=\mathrm{Ker}(\pi)$. Then the result follows from Theorem \ref{main3}.
\end{proof}

\section*{acknowledgements}
The research of the second author was in part supported by a grant from IPM (No. 1400170417).

\end{document}